\documentclass[12pt]{extarticle}
\usepackage{amsmath, amsthm, amssymb, color}
\usepackage[colorlinks=true,linkcolor=blue,urlcolor=blue]{hyperref}
\usepackage{graphicx}
\usepackage{caption}
\usepackage{mathtools}
\usepackage{enumerate}
\usepackage{enumitem}
\usepackage{verbatim}
\usepackage{tikz}
\usetikzlibrary{matrix}
\usetikzlibrary{arrows}
\usepackage{algorithm}
\usepackage[noend]{algpseudocode}
\usepackage{caption}
\usepackage[normalem]{ulem}
\usepackage{subcaption}
\oddsidemargin \evensidemargin
\usepackage{makecell}
\usepackage{array}
\newtheorem{theorem}{Theorem}
\newtheorem{proposition}[theorem]{Proposition}
\newtheorem{lemma}[theorem]{Lemma}
\newtheorem{corollary}[theorem]{Corollary}

\theoremstyle{definition}

\newtheorem{remark}[theorem]{Remark}

\newtheorem{example}[theorem]{Example}

\numberwithin{theorem}{section}

\newcommand{\PP}{\mathbb{P}}
\newcommand{\RR}{\mathbb{R}}
\newcommand{\QQ}{\mathbb{Q}}
\newcommand{\CC}{\mathbb{C} }

\newcommand{\NN}{\mathbb{N}}

\newcommand{\p}{\partial}

\DeclareMathOperator{\im}{Im}

\DeclareMathOperator{\Hom}{Hom}
\DeclareMathOperator*{\Ass}{Ass}

\DeclareMathOperator{\Sol}{Sol}
\DeclareMathOperator{\Ker}{Ker}

\DeclareMathOperator{\Gr}{Gr}

\algrenewcommand\algorithmicrequire{\textbf{Input:}}
\algrenewcommand\algorithmicensure{\textbf{Output:}}

\title{\bf Making Waves}

\author{Marc H\"ark\"onen, Jonas Hirsch and Bernd Sturmfels}

\date{}

\begin{document}
\maketitle

\begin{abstract}
\noindent
We study linear PDE constraints for
vector-valued functions and distributions.
Our focus lies on
wave solutions. These
 give rise to distributions
with low-dimensional support. 
Special waves
from vector potentials are represented by
syzygies. We parametrize all waves
by projective varieties derived from the support of the PDE.
These include determinantal varieties and Fano varieties,
and they generalize wave cones in analysis.
 \end{abstract}

 \section{Introduction}

We consider \emph{homogeneous}, \emph{linear} systems of partial differential equations with \emph{constant} coefficients. In what follows, the acronym PDE 
 refers to such systems. Fusing notational conventions from \cite{AHS21, CS}
 and \cite{ADHR19, DPR}, we study PDE for unknown functions  $\phi \colon \RR^n \to \CC^k$. 
 If the system has $\ell $ constraints then we write the PDE as an
 $\ell \times k$ matrix $A = (a_{ij}) $ with entries in the polynomial ring $R=\CC[\p_1,\dotsc,\p_n]$.
 We  assume that all entries $a_{ij}$ are homogeneous of the same degree $d$,
 so $d$ is the order of our PDE. 
The action of $A$ on $\phi = \phi(x)$ results in a column vector $A \bullet \phi$ of length $\ell$.
Its $i$th entry equals $ \sum_{j=1}^k a_{ij} \bullet \phi_j $, where
$a_{ij}(\partial)$ acts on $\phi(x)$ via $\p_s = \frac{\partial}{\partial x_s}$.
The $\ell$ rows of $A$ generate a submodule $M =  \im_R(A^T)$
of the free module $R^k$. 

The solutions $\phi$ to $A$ depend only on the module $M$.
They are drawn from appropriate spaces of  distributions
 \cite{hormander_vol_I}. The spaces we need will be reviewed in Section \ref{sec2}.
The {\em support} of  $M$ is the variety
defined by the $k \times k$ minors of $A$.
If the matrix $A$ has rank less than $k$, for instance if $\ell < k$, then
these minors vanish identically.
We are here interested in such cases.

\begin{example}[$n=3, k=4, \ell=2$] \label{ex:eins}
We illustrate our setup for the PDE 
\begin{equation} \label{ex:einsA}
A \quad = \quad \begin{bmatrix}
\p_1 & \p_2 & \p_3 & 0 \\
0 & \p_1 & \p_2 & \p_3 
\end{bmatrix}.
\end{equation}
Solutions to $A$ are functions $\phi : \RR^3 \rightarrow \CC^4$ that satisfy
$A \bullet \phi = 0$. In coordinates, this reads
  \begin{small} $$
\frac{\partial \phi_1}{\partial x_1} +
\frac{\partial \phi_2}{\partial x_2} +
\frac{\partial \phi_3}{\partial x_3} \,\,= \,\,
\frac{\partial \phi_2}{\partial x_1} +
\frac{\partial \phi_3}{\partial x_2} +
\frac{\partial \phi_4}{\partial x_3} \,\,= \,\, 0. 
$$ \end{small}
To find some solutions, we write the kernel of $A$ as the image 
 of the skew-symmetric matrix
$$
B \quad = \quad \begin{small}
\begin{bmatrix} 
0 &  -\p_3^2 &  \p_2 \p_3 &\p_1 \p_3  -\p_2^2\, \\
\p_3^2 &  0 &  -\p_1 \p_3 &  \p_1 \p_2 \\
-\p_2 \p_3 &      \p_1 \p_3 &  0 &  -\p_1^2 \\ 
\,\p_2^2-\p_1 \p_3 &  -\p_1 \p_2 &  \p_1^2 &  0
 \end{bmatrix}. \end{small}
$$
  For any function $\psi \colon \RR^3 \to \CC^4$ we get a solution $\phi = B\bullet \psi$, 
  since $A \bullet (B \bullet \psi) = (A B) \bullet \psi = 0$.
  We call $\psi$ a \emph{vector potential}.
For instance, if $\varphi$ is any scalar function and $\psi = \varphi \cdot e_1$, then
\begin{equation} 
\label{eq:psi}
\phi(x) \,\,\,= \,\,\,
(B \bullet \psi)(x) \,\, = \,\, 
(B e_1 \bullet \varphi)(x) \,\, = \,\,\begin{small} 
\biggl[ \, 0 \,\,,\, \frac{\partial^2 \varphi}{ \partial x_3^2} \, , \,
-\frac{\partial^2 \varphi}{ \partial x_2 \partial x_3}\, , \,
\frac{\partial^2 \varphi}{ \partial x_2^2} -\frac{\partial^2 \varphi}{ \partial x_1 \partial x_3}\,
\biggr]^{\! T}. \end{small}
\end{equation}
 The matrix $B$  has rank two, so its column
span over the field $K = \CC(\p_1,\p_2,\p_3)$ has a basis of
  two vectors in $K^4$. For readers of \cite{AHS21} we 
 note that $M = \im_R(A^T) $ 
is a $\{0 \}$-primary submodule of multiplicity $2$ in 
$R^4 $. The {\tt Macaulay2} command {\tt solvePDE} outputs 
two basis vectors. \hfill $\triangle$
\end{example}

This project arose from our desire to understand the hierarchy of
 wave cones in \cite{ADHR19}. These are subsets of $\CC^k$
 which play an important role  in the regularity theory of PDE. 
Recent work on the algebraic side in \cite{CC, CS}
pioneered {\em differential primary decompositions}, where modules
$M$ are encoded by their associated primes along with
certain vectors of length $k$, known as
{\em Noetherian operators}. Our presentation connects these threads
  from analysis and algebra. Along the way, we give an exposition of
known results from control theory \cite{O, oberst_book, shankar99, shankar_notes}.

The waves in our catchy title
are special solutions to the PDE. We develop
geometric theory and algebraic algorithms  for finding them.
Our point of departure is the {\em simple wave}
\begin{equation}
\label{eq:simplewave}
\phi_{\xi,u}(x) \,\,\,= \,\,\, {\rm exp}(i \xi \cdot x) \cdot u.
\end{equation}
Here $\xi \in \RR^n$, $u \in \CC^k$ and $i = \sqrt{-1}$.
The exponential function is applied to the dot product of
 $x = (x_1,\ldots,x_n)$ with the purely imaginary vector $i \xi$,
 resulting in trigonometric functions.
 The real vector $\xi $ is the {\em frequency}, while
 the complex vector $u$ is the {\em amplitude}.
Simple wave solutions  are characterized by a system of
$\ell$ polynomial equations in their $n+k$ coordinates:
\begin{equation}
\label{eq:Au0}
 A \bullet  \phi_{\xi,u} = 0 \quad \hbox{if and only if}
\quad
A(\xi) \cdot u = 0 . 
\end{equation}
For the PDE in Example \ref{ex:eins}, 
setting $\varphi (x) = {\rm exp}( i \xi \cdot x)$ in (\ref{eq:psi}) yields a
simple wave solution. 

Our standing assumption is that
 all entries of the matrix $A$ are homogeneous polynomials in $\p_1,\ldots,\p_n$ of the same degree $d$.
 This implies that
$A$ is elliptic---therefore smoothing---if and only if there are no nontrivial wave solutions. 
Therefore, the existence of wave solutions has a major impact on the analytical properties of the operator, cf.~\cite[Chapter 2.1]{schechter77}.

Wave solutions are obtained from
superpositions of simple waves and taking limits.
In the superpositions we allow here, the amplitude $u$ is fixed,
whereas the frequency $\xi $ runs over linear subspaces of $\RR^n$
all of whose points satisfy (\ref{eq:Au0}). Taking limits of such superpositions leads to
waves that are distributions with small support. This construction
will be explained in detail in Section~\ref{sec2},
where we also review  basics on 
 spaces of functions and distributions.

In Section \ref{sec3} we construct 
solutions that arise from a vector potential.
If $M$ is $\{0 \}$-primary then all solutions to $A$ have that property.
This is based on results in commutative algebra and control theory,
notably due to Shankar  \cite{shankar99, shankar_notes}.
It can be used to build distributional solutions that are
compactly supported and act like
waves in the interior of their support.

In Section \ref{sec4} we turn to algebraic geometry,
and we introduce  projective varieties that
parametrize wave solutions. These generalize
the determinantal varieties of matrices of linear forms.
In Section \ref{sec5} we examine the analytic meaning
of wave varieties and obstruction varieties, and 
discuss the analytic implications of working algebraically in
complex projective spaces.
In Section \ref{sec6} we introduce varieties of wave pairs. These
generalize Fano varieties \cite[Example 6.19]{Harris} inside Grassmannians.
We present methods for computing
wave pairs and wave varieties, and we illustrate these
on several examples. In the context of a given PDE $A$,
 these scenarios give interesting
distributional solutions to $A$ 
with low-dimensional support.

We close the introduction with a well-known equation from the theory of
 elasticity~\cite{GP, muskhelishvili}.

\begin{example}[Saint-Venant's tensor]
Set $d=2,k  = \binom{n+1}{2},\ell = k^2$, and identify $\CC^k$ with the space of symmetric $n \times n$ matrices.
We consider
 matrix-valued distributions $\phi: \RR^n \rightarrow \CC^k$.
The {\em Saint-Venant operator} $A$ characterizes the kernel of the 
$2$-dimensional X-ray transform:
\begin{equation}
\label{eq:SV}
  A \bullet \phi \,\,=\,\, \bigl( \p_i\p_j \phi_{ab} + \p_a\p_b \phi_{ij} - \p_i\p_a \phi_{jb} - \p_j \p_b \phi_{ia} \bigr)_{i,j,a,b = 1,\dotsc,n}.
\end{equation}
In our notation, $A$ is an $\ell \times k$ matrix whose nonzero entries are quadratic monomials $\p_i \p_j$.
By removing redundant rows, using \cite{GP}, the number
of rows of $A$ can be reduced to 
$\ell = \frac{1}{6} \binom{n^2}{2}$.
The PDE $A$ has a  vector potential $B$, as in Section \ref{sec3}, given by the symmetric gradient:
$$
  B \bullet \psi \,\,=\,\, \bigl( \, \partial_i \psi_j + \partial_j \psi_i\,\bigr)_{i,j = 1,\dotsc,n}.
$$
The wave pair variety $\mathcal{P}_A^{n-1}$ of Section \ref{sec6} lives in the space
 $\PP^{n-1} \times \PP^{k-1}$, and it coincides with
the incidence variety $\mathcal{I}_A$ in
(\ref{eq:IA}). It is defined by the
   $3 \times 3$ minors of  the $(n+1) \times (n+1)$-matrix
\begin{equation}
\label{eq:augintro}
    \begin{bmatrix}
      0 & y_1 & y_2 & \cdots & y_n \\
      y_1 & z_{11} & z_{12} & \cdots & z_{1n} \\
      y_2 & z_{12} & z_{22} & \cdots & z_{2n} \\
      \vdots & \vdots & \vdots & \ddots & \vdots \\
      y_n & z_{1n} & z_{2n} & \cdots & z_{nn}
    \end{bmatrix}.
\end{equation}
 The wave variety $\mathcal{W}_A \subset \PP^{k-1}$ of Section \ref{sec4} 
 is given by the $3 \times 3$ minors of the $n \times n$ matrix~$(z_{ij})$.
  This example is a variant of the curl operator in Proposition~\ref{prop:curl},
with (\ref{eq:augintro}) playing the role of (\ref{eq:augmentedmatrix}).
 It underscores the relevance of nonlinear algebra \cite{MS} for the physical sciences.
 \hfill $\triangle$
\end{example}

\section{Spaces and Waves}
\label{sec2}

In real analysis it is customary to consider functions from the
real space $\RR^n$ to the complex space $\CC^k$. By
passing to the real part or the imaginary part, one
obtains a function from $\RR^n$ to $\RR^k$.
For instance, the real part of the simple wave
(\ref{eq:simplewave}) is an expression in terms of
sine and cosine.
If $\phi$ satisfies a linear PDE $A$ with real coefficients
then its real part satisfies $A$.

We now fix the setting found in standard analysis texts,
such as H\"ormander's book \cite{hormander_vol_I}.
Fix an open convex subset $\Omega \subset \RR^n$.
Let $C_c^\infty(\Omega,\CC^k)$ denote the space 
of smooth functions $f \colon \RR^n \to \CC^k$
whose support is compact and contained in $\Omega$.
One writes $\mathcal{D} = C_c^\infty(\Omega,\CC^k)$
for the same space after endowing it with 
the topology of sequential convergence.
Here, $\{f_k\}$ converges to $f$ in $\mathcal{D}$
if there exists a compact set $K \subset \Omega$,
containing the supports of $f$ and of each $f_k$,
such that the norm of 
$\,\partial^\alpha  f_k - \partial^\alpha f\,$ converges
to zero for each $\alpha \in \NN^n$.
Thus, $\mathcal{D}$ is a topological vector space
over $\CC$ which is a module over
$R = \CC[\p_1,\ldots,\p_n]$ by differentiation.

The space of distributions $\mathcal{D}'$ is a subspace of the
vector space dual to $\mathcal{D}$. Its elements
are the linear functions $\mathcal{D} \rightarrow \CC$ that are continuous in the topology above.
This can be expressed by the following growth condition.
  A \emph{distribution} is a linear functional $\phi \colon \mathcal{D} \to \mathbb{C}$ such that,
   for every compact set $K \subset \mathbb{R}^n$, there exist
   positive real constants $C$ and $ d$ with
$$
    |\phi(f)| \,\,\leq \,\,C \sum_{|\alpha| \leq d} \sup |\partial^\alpha f| \qquad
\hbox{for all $f \in \mathcal{D}$ with ${\rm supp}(f) \subseteq K.$}
  $$
  We denote by $\mathcal{D}'$ the space of distributions from
  $\Omega$ to $\CC^k$.
  Every compactly supported  smooth function is also a distribution.
  Namely, there is a natural inclusion from $\mathcal{D}$ into $\mathcal{D'}$
which takes a function  $\phi : \RR^n \to \CC^k$ to 
the linear functional $\mathcal{D} \rightarrow \CC, f \mapsto \int f \cdot \phi$ for all $f \in \mathcal{D}$.
Other prominent examples of distributions are the Dirac delta functions and the step functions.

We now come to the special role played by exponential functions.
For this, we introduce the {\em Schwartz space}
$\mathcal{S} = \mathcal{S}(\RR^n,\CC^k)$ whose elements are
smooth functions $f $ such that
$||x^\beta \partial^\alpha f||_\infty $ is finite for all $\alpha,\beta \in \NN^n$.
This space includes the simple waves (\ref{eq:simplewave}), since the coordinates
of $\xi$ are real. However,
many nice functions, such as polynomials, are not in $\mathcal{S}$.

Most relevant for us is that  $\mathcal{D} = C^\infty_c(\Omega,\CC^k)$
 is a subspace of $\mathcal{S} = \mathcal{S}(\RR^n,\CC^k)$.
The key feature of the Schwartz space is the endomorphism
known as the {\em Fourier transform}
$\,\hat{} : \mathcal{S} \rightarrow \mathcal{S}$.
By applying $\,\,\hat{}\,\,$ twice, we see that every function in $\mathcal{S}$
admits an integral representation
\begin{equation}
\label{eq:fourier} f(x) \,\,\, = \,\,\,
\int_{\RR^n} \!
{\rm exp}(\,2 \pi i \,\xi \cdot x)\, {\hat f}(\xi) \,d \xi. 
\end{equation}
The dual to the Schwarz space consists of the {\em tempered distributions}.
We have inclusions
\begin{equation}
\label{eq:tempered}
 \mathcal{D}\, \hookrightarrow \,\mathcal{S} \,\hookrightarrow\,
\mathcal{S}' \,\hookrightarrow \, \mathcal{D}'.
\end{equation}
All of these spaces are $R$-modules because the linear
map $\partial^\alpha : \mathcal{D} \rightarrow \mathcal{D}$
is continuous, so we get a dual
$ (\partial^\alpha)^* : \mathcal{D}' \rightarrow \mathcal{D}'$
which restricts to $\mathcal{S}'$ and $\mathcal{S}$.
One subtle issue here is the sign. The action of $R = \CC[\p_1,\ldots,\p_n]$
on distributions by taking partial derivatives satisfies
$$ (\partial_i \bullet \phi) (f)\,\, = \,\,-\phi \bigl( {\partial f/ \partial x_i}\bigr) . $$

The notion of waves used in this paper arises from
superpositions of the simple waves~(\ref{eq:simplewave}),
\begin{equation}
\label{eq:superposition} \phi(x) \,\,=\,\, \sum_{j=1}^p \lambda_j \phi_{\xi_j,u} (x) . 
\end{equation}
Here the amplitude $u \in \CC^k$ is fixed but
 the frequencies $\xi_1,\ldots,\xi_p $ vary in $  \RR^n$.
The coefficients $\lambda_1,\ldots,\lambda_p$ are
complex numbers. If each summand in (\ref{eq:superposition}) satisfies $A$ then so does~$\phi(x)$.
The integral representation (\ref{eq:fourier}) of Schwartz functions 
by the Fourier transform implies that
every distribution $\delta \in \mathcal{D}'$ can be approximated by a
 sequence of waves $\phi^{(1)}, \phi^{(2)}, \ldots $ of the form~\eqref{eq:superposition}.
 
\begin{lemma} \label{lem:expdense}
The linear span of the exponential functions $x  \mapsto {\rm exp}(i\xi \cdot x)$ is dense in $\mathcal{D}'$.
\end{lemma}

Our aim is to create interesting distributions by taking limits of waves (\ref{eq:superposition})
in $\mathcal{D}'$.
To this end, suppose that $\xi_1,\ldots,\xi_p$ span a 
linear subspace $\pi$ of $\RR^n$ such that
$A \bullet \phi_{\xi,u} = 0$ for all $\xi \in \pi$.
We then consider the closure in $\mathcal{D}'$
of the space of all waves (\ref{eq:superposition}) whose
frequencies $\xi_j$ are in $\pi$. Each element in that closure
satisfies the PDE $A$, and the closure contains
distributional solutions with small support. This motivates the following definitions.
As before, $A \in R^{\ell \times k}$ is a matrix whose entries are homogeneous of degree $d$. A {\em wave pair} for $A$ is a pair $(u, \pi)$, where
$u \in \CC^k$ and $\pi$ is a linear subspace of $\RR^n$, such that
$A(\xi) u = 0$ for all $\xi \in \pi$. If  $(u,\pi)$ is a wave pair then any
superposition (\ref{eq:superposition}) with $\xi_1,\ldots,\xi_p \in \pi$ is  a
{\em classical wave solution} of $A$. 
A {\em wave solution} to $A$ is any
distribution  in the closure in $\mathcal{D}'$
of the classical wave solutions.

\begin{proposition}\label{prop:wave_sols}
Consider any wave pair $(u,\pi)$ for the operator $A$ and set
  $r = {\rm codim}(\pi)$. 
 The associated wave solutions are 
precisely the distributions of the form 
\begin{equation}
 \label{eq:wavesolution}
\qquad  \phi (x) \,=\,\delta\bigl(L_1(x),\ldots,L_{n-r}(x) \bigr) \cdot u,
\end{equation}
where $L_1, \dotsc, L_{n-r}$ are linear form satisfying $ \pi^\perp = \{x \in \RR^n: L_1(x) = \cdots = L_{n-r}(x) = 0\}$ and $\delta$ is any distribution in $\mathcal{D}'(\RR^{n-r},\CC)$. 
Thus, equation \eqref{eq:wavesolution} characterizes wave pairs as follows: if $\phi(x)$ is a solution to the PDE $A$ for all $\delta \in \mathcal{D}'(\RR^{n-r},\CC)$ then $(u, \pi)$ is a wave pair.
 \end{proposition}

 \begin{remark}\label{rmk:composition}
   The notation $\delta(L\cdot) := \delta(L_1(x), \dotsc, L_{n-r}(x))$ refers to an
      extension from smooth functions to distributions. Following \cite[Chapter 6]{hormander_vol_I},
   one can define it as follows.
   Given a real matrix $L \in \RR^{(n-r)\times r}$, fix an orthonormal basis $v_1,\dotsc, v_r$ for $\ker(L)$ and let $L'\in \RR^{r\times n}$ be the matrix with the $v_i$ as rows.
   The matrix $H = \begin{bmatrix}
     L \\ L'
     \end{bmatrix}$ defines an endomorphism $\RR^n \to \RR^n,\,x \mapsto (y',y'')$, where $y' \in \RR^{n-r}$, $y'' \in \RR^r$.
     Its inverse is $H^{-1} = \begin{bmatrix}
   L^T(LL^T)^{-1} & L'^T
   \end{bmatrix}$.
If $\delta \colon \RR^{n-r} \to \CC$ is smooth,
   then,  by a change of variables, for any test function $f \in \mathcal{D}(\RR^n, \CC)$,
$$          \begin{matrix}
     \delta(L \cdot)(f) \,=\, \int_{\RR^n} \delta(L(x)) f(x) \,dx \,=\, \int_{\RR^n} \delta(y') f(H^{-1}y) |\det(H^{-1})| \,dy \qquad 
  \qquad    \smallskip \\ \qquad \quad
      \,=\, \frac{1}{\sqrt{\det(LL^T)}}\int_{\RR^n} \delta(y') f(H^{-1}y) \,dy 
      \,=\, \frac{1}{\sqrt{\det(LL^T)}}\int_{\RR^{n-r}} \delta(y') \int_{\RR^r} f(H^{-1}y) \, dy''\,dy'.
   \end{matrix} $$
   We write this as
 $\,    \delta(L\cdot)(f) = \delta(1(F))$, with the constant function
  $1 : \RR^r \to \CC$  and
       $F(y) = \frac{1}{\sqrt{\det(LL^T)}}f(H^{-1}y)$.
   There exists a unique distribution $\delta \otimes 1$ such that $(\delta \otimes 1)(F) = \delta(1(F)) = 1(\delta(F))$ for all $F \in \mathcal{D}(\RR^n,\CC)$.
   Now define $\delta(L\cdot)(f) := (\delta \otimes 1)(F)$ for \emph{arbitrary} distributions~$\delta$.
    \end{remark}

 \begin{proof}[Proof of Proposition~\ref{prop:wave_sols}]
    Write $x = (x_1,\dotsc,x_n)$ and $y = (y_1,\dotsc,y_{n-r})$ for the coordinates of $\RR^n$ and $\RR^{n-r}$, and let $L$ denote the $(n-r) \times n$ matrix of coefficients of  $L_1,\dotsc, L_{n-r}$.
    For $\eta \in \RR^{n-r}$, consider the wave function $x \mapsto \delta_\eta(Lx)\cdot u$ associated with the exponential function $\delta_\eta(y) = \exp(i\eta \cdot y)$.
    Applying the differential operator $A$ to that wave function yields
    \begin{align}\label{eq:pde_on_exponential}
      A  \bullet (\delta_\eta( Lx) \cdot u) \, =  \,
      A \bullet ( {\rm exp}(i\eta Lx) \cdot u) \,=\,i^d\,
      {\rm exp}(i\eta L x) 
      \cdot A(\eta L) u.
    \end{align}
    This vector of length $\ell$ is zero for all $\eta \in \RR^{n-r}$ if and only if
    $(u, \pi)$ is a wave pair. Since the space spanned by the exponential functions
    $\delta_\eta$ for $\eta \in \RR^{n-r}$ is dense in the space of all 
    distributions, by Lemma \ref{lem:expdense}, the first assertion follows.
    
    The second statement
	follows from the fact that $A\cdot (\delta_\eta(Lx)\cdot u)=0$ for all $\eta$ if and only if $(u,\pi)$ is a wave pair,
	together with the simple observation that $\delta_\eta \in \mathcal{D}'(\RR^{n-r},\CC)$.
\end{proof}

We seek wave pairs $(u,\pi)$ where $r$ is as small as possible.
Indeed, if $r$ is small then we can build distributional solutions
with small support.  What follows is the standard construction.

\begin{remark}
Let $\delta $ be the Dirac delta distribution at the origin in $\RR^{n-r}$.
The distribution~$\phi$ in (\ref{eq:wavesolution}) is supported
on the $r$-dimensional  subspace $\pi^\perp$ of $\RR^n$.
If $(u,\pi)$ is a wave pair then $\phi$ satisfies the PDE $A$.
Such {\em $A$-free measures} are important in the calculus
of variations~\cite{ADHR19, DPR}.
\end{remark}

\begin{example}[$n=3, k=4, r=2$] \label{ex:zwei}
Fix the matrix $A$ in Example \ref{ex:eins}.
For all $\xi \in \CC^3 \backslash \{0\}$, 
the linear space $\,{\rm ker} \,A(\xi)$ has dimension $2$. It
consists of all vectors $u \in \CC^4$ such that
\begin{equation}
\label{eq:xiu}
\begin{bmatrix}
\xi_1 & \xi_2 & \xi_3 & 0 \\
0 & \xi_1 & \xi_2 & \xi_3  \end{bmatrix}
\begin{bmatrix} u_1 \\ u_2 \\ u_3 \\ u_4 \end{bmatrix} \,\, = \,\,
\begin{bmatrix}
u_1 & u_2  & u_3 \\ u_2 & u_3 & u_4 \end{bmatrix}
 \begin{bmatrix} \xi_1 \\ \xi_2 \\ \xi_3 \end{bmatrix} \,\, = \,\,
 \begin{bmatrix} 0 \\ 0 \end{bmatrix}. 
\end{equation}
This equation characterizes simple waves $\phi_{\xi,u}$ that satisfy $A$.
With $r=2$  in (\ref{eq:wavesolution}) we can take
\begin{equation}
\label{eq:L_1}
\qquad L_1(x) \,\,=\,\, (u_3^2-u_2 u_4)x_1 + 
 (u_1u_4-u_2 u_3) x_2
 + (u_2^2-u_1 u_3) x_3 
 \qquad \hbox{for any $u \in \RR^4$.} 
\end{equation}
 By superposition we obtain waves with $r=1$.
 Here $u $ must be chosen such that
 the three coefficients in (\ref{eq:L_1}) vanish.
 This means that $u$ lies in  the cone over
the {\em twisted cubic curve}:
\begin{equation}
\label{eq:twistedcubic}
 (u_1,u_2,u_3,u_4) \,\, = \,\,
(s^3,s^2t, st^2,t^3).
\end{equation}
This is the wave variety $\mathcal{W}^1_A  \subset \PP^3$ in
Example \ref{ex:dadazwei}.
We obtain wave pairs $(u,\pi)$ with ${\rm codim}(\pi) = 2$, 
and thus solutions supported on a plane in $\RR^4$, by taking
 the two linear forms
$$ L_1(x) \,=\, t x_1 - s x_2 \quad {\rm and} \quad
     L_2(x) \,=\, t x_2 - s x_3. $$
    Indeed, 
          $\phi(x) = \delta \bigl(L_1(x),L_2(x) \bigr) \cdot u$ is a wave solution of $A$,
     for any
$\delta \in \mathcal{D}'(\RR^2,\CC)$.
\hfill
$ \triangle $
\end{example}

\section{Syzygies}
\label{sec3}

Our task is to solve the PDE $A\bullet \phi = 0$.
We now focus on solutions that are represented by
a vector potential $\psi$ as in Example~\ref{ex:eins}.
This section is independent from the rest of the paper.
It offers tools for constructing compactly supported solutions
with desirable properties, including those that
are waves in the interior of their support.
From an analytic point of view, the existence of compactly supported solutions is of interest in the context of understanding the space of Young measures; see e.g.~\cite{KR1}.
We note that,
while the functions in $\mathcal{D}$ 
are not waves, we can fuse them with waves to create
solutions that are of interest in fields such as convex integration \cite{KMS, Pompe}.
The choice of  the potential $\psi$ allows for~this.

Let $\mathcal{F}$ be a space of functions or distributions
$\RR^n \rightarrow \CC$ such as those in
(\ref{eq:tempered}), or those in \cite{O, oberst_book, shankar99, shankar_notes}.
We assume that $\mathcal{F}$ is an $R$-module under differentiation.
Solving the PDE means describing all $k$-tuples $\phi \in \mathcal{F}^k$  with
 $A\bullet\phi = 0$. 
The set of  all solutions is the $R$-module
\begin{align*}
  \Sol_{\mathcal{F}}(A) \,:= \, \Ker_{\mathcal{F}}(A) \,\,\subset \,\, \mathcal{F}^k.
\end{align*}
Algebraic algorithms for this task, along with implementations in {\tt Macaulay2},
are presented in \cite{AHS21, NoetherianOperators_package}.
We note that $\phi \in \Sol_{\mathcal{F}}(A)$ if and only if $m \bullet \phi = 0$ for all $m \in M $,
where $M$ is the submodule of $R^k$ spanned by the rows of $A$.
Thus the solution space depends only on the
rowspan of $A$.
More formally, we apply
the functor $\Hom_R(\,\, \cdot \,\, , \mathcal{F})$ to the exact sequence
\begin{align*}
  R^\ell \,\xrightarrow{A^T} \,R^k \,\rightarrow\, R^k/M \,\rightarrow \, 0.
\end{align*}
The result of this dualization step is the sequence
$
  \mathcal{F}^\ell \,\xleftarrow{A} \,\mathcal{F}^k \,\leftarrow 
  \, \Hom_R(R^k/M, \mathcal{F})\, \leftarrow\, 0$.
This sequence is also exact, so our solution space can be written as
an $R$-module as follows:
\begin{align}\label{eq:sol_cong_hom}
  \Sol_{\mathcal{F}}(A) \,=\,
  \Sol_{\mathcal{F}}(M) \,\cong \,\Hom_R(R^k/M, \mathcal{F}).
\end{align}
Note that $\Sol_{\mathcal{F}}(\,\, \cdot \,\,)$ is inclusion reversing: if $M_0 \subseteq M_1$ are submodules of $R^k$, then $\Sol_{\mathcal{F}}(M_0) \supseteq \Sol_{\mathcal{F}}(M_1)$.
To question to what extent a module $M$ of PDE can be recovered from
its solution spaces is addressed by the Nullstellensatz in
\cite{shankar99}. Indeed, there is a considerable body of literature in
control theory on the connection between linear PDE and
commutative algebra. We here follow the expositions by
Oberst  \cite{O, oberst_book} and Shankar \cite{shankar_notes},
and the references therein.

Recall that an
$R$-module $\mathcal{F}$ is \emph{injective} if the functor $\Hom(\,\, \cdot \,\,, \mathcal{F})$ is exact.
An $R$-module $\mathcal{F}$ is an \emph{injective cogenerator} (cf.~\cite{O}) if the following holds:
  a sequence of $R$-modules
  $A \to B \to C$ is exact if and only if $\Hom_R(A, \mathcal{F}) \leftarrow \Hom_R(B, \mathcal{F}) \leftarrow 
  \Hom_R(C,\mathcal{F})$ is exact.
The following result concerns the space $\mathcal{D}'$
of distributions. It fails for the subspace $\mathcal{D}$ of $\mathcal{D}'$.

\begin{proposition}[{\cite[Corollary 4.36]{O}}]\label{thm:inj}
  The $R$-module $\mathcal{D}'$ is an injective cogenerator, so
    the inclusion reversing map from
      submodules $M \subseteq R^k$ to
        solution spaces $\Sol_{\mathcal{D}'}(M)$ is bijective.
\end{proposition}

Our goal is to compute the subspace of solutions to $A$ that
are derived from vector potentials as in (\ref{eq:psi}).
This is usually a proper subspace, as seen in the following simple~example.

\begin{example}[$k=\ell=n=d=2$] \label{ex:222}
Let $M$ be the $R$-module generated by the rows of
  \begin{align*}
    A \,=\, \begin{bmatrix}
      \partial_1^2 & \partial_1 \partial_2 \\
      \partial_1 \partial_2 & \partial_2^2
    \end{bmatrix}.
  \end{align*}
The solutions come in two flavors, corresponding to a primary decomposition
$M = M_0 \,\cap \,M_1$. Namely, ${\rm Sol}_\mathcal{F}(M) =
{\rm Sol}_\mathcal{F}(M_0) + {\rm Sol}_\mathcal{F}(M_1)$, where
  $M_i$ is the module generated by the rows~of
  $$
    A_0 \,=\,  \begin{bmatrix}
      \partial_1 & \partial_2
    \end{bmatrix} \quad {\rm and} \quad
    A_1 \,=\, \begin{bmatrix}
      \partial_1^2 & \partial_1\partial_2 & \partial_2^2 & 0 & 0 & 0\\
      0 & 0 & 0 & \partial_1^2 & \partial_1\partial_2 & \partial_2^2
    \end{bmatrix}^T.
    $$
The solutions to the PDE $A_0$ are
     $\begin{bmatrix*}[r] - \partial_2 \bullet \psi\\ \partial_1 \bullet \psi\end{bmatrix*}$
     for any $\psi \in \mathcal{F}$, while the solutions to  $A_1$ are 
      $\begin{bmatrix} ax_1+bx_2+c \\ a' x_1 + b'x_2 + c' \end{bmatrix}$ for
       $a,a',b,b',c,c' \in \CC$.
Both the minimal prime $\{0\}$ and the embedded prime $\langle \partial_1,\partial_2 \rangle$
have multiplicity one in $M$. We can assume
 $b=c=a'=b'=c'=0$, after a suitable choice of $\psi$. This is seen from the output of
     {\tt solvePDE} in {\tt Macaulay2} \cite{AHS21, M2, NoetherianOperators_package}.
   \hfill
$ \triangle $
\end{example}

To achieve our goal for an arbitrary matrix $A \in R^{\ell \times k}$, we compute
a matrix $B$ such that 
\begin{align}\label{eq:syz}
  R^{k'} \xrightarrow{B} R^k \xrightarrow{A} R^\ell
\end{align}
is an exact sequence.
The columns of $B$ are syzygies of $A$.
The transpose of this sequence~is 
\begin{align}\label{eq:syz_T}
  R^{k'} \xleftarrow{B^T} R^k \xleftarrow{A^T} R^\ell.
\end{align}
This is a complex but it is generally not exact.
Applying $\Hom_R(\,\, \cdot \,\,, \mathcal{F})$, we get the complex
\begin{align}\label{eq:hom_syz}
  \mathcal{F}^{k'} \xrightarrow{B} \mathcal{F}^{k} \xrightarrow{A} \mathcal{F}^\ell,
\end{align}
and hence $\im_\mathcal{F} (B) \subseteq \Sol_\mathcal{F}(A)$.
This means that $B \bullet \psi$ is a solution to our PDE  $A$
for any $\psi \in \mathcal{F}$.
If the equality  $\im_\mathcal{F} (B) = \Sol_\mathcal{F}(A)$
holds then we say that $A$
{\em admits a vector potential}.
This was the case in Example \ref{ex:eins}
but not in Example \ref{ex:222}, where (\ref{eq:syz_T}) is not exact.

We briefly recall some definitions.
An element $f$ in an $R$-module $U$
is a \emph{torsion element} if 
$rf = 0$ for some $r \in R \backslash \{0\}$.
The \emph{torsion submodule} of $U$ is the module of torsion elements.
The module $U$ is \emph{torsion} if it is equal to its torsion submodule.
The module $U$ is \emph{torsion-free} if its torsion submodule is zero.
A prime ideal $P \subset R$ is 
\emph{associated} to $U$ if there is some element $u \in U$
that $P := \{r \in R \colon ru =0 \}$.
The set of associated primes of $U$ is denoted $\Ass(U)$.
The module  $U$ is called \emph{$P$-primary} if $\Ass(U) = \{P\}$.
By standard abuse of language, we use these adjectives for 
a submodule $M \subset R^k$ when
$U = R^k/M$ has that~property.

\begin{theorem} \label{thm:torsion_free}
  Suppose that the sequence \eqref{eq:syz} is exact. Then the following are equivalent:
  \begin{itemize}
  \item[(1)] The PDE $A$ admits a vector potential, i.e.~the sequence (\ref{eq:hom_syz}) is exact.
  \item[(2)] The sequence \eqref{eq:syz_T} is exact.
  \item[(3)] The module $M = \im_R(A^T)$ is torsion-free.
  \item[(4)] The module $M = \im_R(A^T)$ is $\{0\}$-primary.
  \end{itemize}
\end{theorem}

\begin{proof}
The equivalence of (1) and (2) holds because $\mathcal{F}$ is
an injective cogenerator. The equivalence of (2), (3), (4)
is a standard result in commutative algebra, also found in~\cite{shankar_notes}.
\end{proof}

If the conditions in Theorem \ref{thm:torsion_free} are met then
we have a parametrization of all solutions:
\begin{equation}
\label{eq:solutionsareimages}
\,\Sol_\mathcal{F}(A) \,\,=\,\, \im_{\mathcal{F}}(B) \,=\, \bigl\{B \bullet \psi \colon \psi \in \mathcal{F}^{k'}\bigr\}.
\end{equation}

In general, 
$\Sol_\mathcal{F}(A) $ is strictly contained in $\im_{\mathcal{F}}(B)$: not all
 solutions of $A$ are in the image of~$B$.
In that case, the operator can be split into two operators $A_0$ and $A_1$, where $A_0$ admits a vector potential and $A_1$ does not,
  in the following sense: for all $B \in R^{k \times k'}$, there 
 exists $\psi \in \mathcal{F}^{k'}$ such that $B \bullet \psi \not \in \Sol_{\mathcal{F}}(A_1)$.
This condition is equivalent to $\{0\} \not\in \Ass(M_1)$ for $M_1 = \im_R(A_1^T)$.
It is also equivalent to $R^k/M_1 $ being a torsion module.

We write $M = M_0 \cap M_1$, where $M_0$ is $\{0\}$-primary, and $\{0\} \not\in \Ass(R^k/M_1)$.
This is obtained from a primary decomposition of $M$, where 
$M_1$ is the intersection of all primary components that are not $\{0\}$-primary.
The solutions satisfy $\Sol_\mathcal{F}(M) =
 \Sol_\mathcal{F}(M_0) + \Sol_\mathcal{F}(M_1)$.
This is known in control theory~\cite{shankar_notes}
as the \emph{controllable-uncontrollable decomposition}.

\begin{theorem} \label{thm:hascompactly}
The PDE $A$ has
    compactly supported solutions if and only if $\{ 0 \} {\in} \Ass(M)$.
\end{theorem}

\begin{proof}
This result is contained in \cite[\S 3]{shankar_notes}. We offer a short proof.
Suppose $\{0\} \in \Ass(M)$. Then $M$ is a submodule of a $\{0\}$-primary module 
$M_0 \subseteq R^k$. Write $M_0$ as the $R$-row span of a matrix $A_0$
  and let $B_1$ be any column in its syzygy matrix $B$.
  Then for any compactly supported distribution $\psi$, the solution $B_1 \bullet \psi \in \Sol_{\mathcal{F}}(A)$ is also compactly supported.

For the converse,  suppose $\{0\} \not\in \Ass(M)$.
Let $\phi \in  \mathcal{F}^k \backslash \{0\}$ be a compactly supported solution.
By Proposition \ref{thm:inj}, there exists $f \in R^k \backslash M$ such that $f^T \bullet \phi \neq 0$.
  Since $R^k/M$ is torsion, $rf \in M$ for some nonzero $r \in R$.
         Thus $r \bullet f^T \bullet \phi = 0$.
  Taking Fourier transforms, by the Paley-Wiener-Schwartz Theorem \cite[Thm.~7.3.1]{hormander_vol_I}, we get the equation $r(\xi) \cdot f(\xi)^T \hat\phi(\xi) = 0$ of analytic functions.
  Since $r(\xi) \neq 0$, we must have $f(\xi)^T \hat\phi(\xi) = 0$, a contradiction.
\end{proof}

In conclusion, for any linear PDE $A$ as above,
 the solution space $\Sol_\mathcal{F}(A)$ decomposes into a subspace
$\Sol_\mathcal{F}(A_0) = \im_\mathcal{F}(B)$ 
which contains all compactly supported solutions, and another
subspace $\Sol_\mathcal{F}(A_1)$, with no compactly supported solutions at all.
We are interested in the former solutions, so our computational task is to 
go from $A$ to $B$ and  then to $A_1$. This amounts to two syzygy computations,
and is easily carried out in {\tt Macaulay2}. 
 By applying~$B$ to potentials
 $\psi(z)  = \delta(L_1(z),\ldots,L_{n-r}(z))\cdot e_i $, we 
 obtain interesting
  solutions to $A$,
   as in~(\ref{eq:wavesolution}).

\section{Varieties}
\label{sec4}

Our  $\ell \times k $ matrix $A$ specifies PDE constraints
of order $d$
for distributions $\phi:\RR^n \rightarrow \CC^k$.
The solution spaces (\ref{eq:sol_cong_hom}) depend
only on the $R$-module $M$ generated by the rows of~$A$.
In this section we introduce several algebraic varieties
that are naturally associated with $A$. As is customary in
algebraic geometry, we work in complex 
projective spaces rather than in real affine spaces.
Every subvariety of $\PP^{k-1}$ corresponds to a cone
in $\CC^k$, which is a complex variety defined by homogeneous equations,
and by restricting to $\RR^k$ one obtains a real cone.
Among such cones are the wave cones from
\cite{ADHR19} which motivated our study.
 We shall return to the
analytic perspective  in the next section. In what follows, however,
we stick to algebra. This means working in the
projective spaces $\PP^{k-1}$ and $\PP^{n-1}$ over the complex numbers~$\CC$.

For any point $y \in \PP^{n-1}$ we write $A(y)$ for the
complex $\ell \times k$ matrix that is obtained from $A$ by
replacing each $\partial_i$ with the coordinate $y_i$.
The matrix $A(y)$ is well-defined up to scale.
 We view it as a point in the projective space $\PP^{\ell k - 1}$. 
We write $z$ for points in $\PP^{k-1}$, and we~set
\begin{equation}
\label{eq:IA} \mathcal{I}_A \,\,\,= \,\,\,
\bigl\{\,(y,z) \in \PP^{n-1} \times \PP^{k-1}\,:\,
A(y) \cdot z = 0 \,\bigr\} .
\end{equation}
This is our algebro-geometric representation of the
relation between frequencies and amplitudes seen in  (\ref{eq:Au0}).
The projection of the incidence variety $\mathcal{I}_A$ onto the first factor equals
\begin{equation}
\label{eq:SA}
 \mathcal{S}_A \,\,\, = \,\,\, \{ \, y \in \PP^{n-1} \,: \, {\rm rank}(A(y)) \leq k-1 \,\}. 
 \end{equation}
This projective variety is the support of our PDE $A$. It depends only on the
module $M$. The notation $V(M)$ was used 
in \cite[Section 3]{AHS21} for the affine cone over $\mathcal{S}_A$.
The role of the support for simple waves was
highlighted in \cite[Lemma 3.6]{AHS21}.
A natural set of polynomials that define $\mathcal{S}_A$
set-theoretically is the
$k \times k$ minors of $A$.
However, these minors usually do not suffice to 
generate the radical ideal of $\mathcal{S}_A$.
There are two interesting extreme cases, namely
$\mathcal{S}_A = \PP^{n-1}$ and $\mathcal{S}_A = \emptyset$.
The former identifies PDE with
compactly supported solutions
 (cf.~Theorem \ref{thm:hascompactly}), while the
latter identifies PDE whose only solutions are polynomials \cite[Theorem 3.8]{AHS21}.

We next consider the projection of the incidence variety $\mathcal{I}_A$
onto the second factor $\PP^{k-1}$. The resulting projective variety is called the 
\emph{wave variety} of $A$, and we write it as follows:
  \begin{align*}
    \mathcal{W}_A \,\,\,:= \,\, \bigcup_{y \in \PP^{n-1} } \ker A(y).
  \end{align*}
The kernel in this definition is a linear subspace of $\PP^{k-1}$,
so $\mathcal{W}_A$ is a projective variety in~$\PP^{k-1}$.
This is the algebraic variant of the
{\em wave cone} considered in analysis; see
\cite[Theorem 1.1]{DPR} and surrounding references. 
We shall return to this in Section \ref{sec5} where it is denoted
  $\mathcal{W}_{A,\RR}$.
  
\begin{example}[$n= k=\ell=3,d=2$] Consider the second order PDE given by the matrix 
$$ A \,\,\, = \,\,\, \begin{bmatrix} 
\phantom{-}\p_1^2 & \phantom{-}\p_2^2 &\,\, \p_3^2 \,\,\,\\
 -\p_2^2 & \phantom{-}\p_3^2 & \,\,\p_1^2\, \,\,\\
  -\p_3^2 & -\p_1^2 &\,\, \p_2^2\,\,\, \end{bmatrix}\! . $$
   Its support $\mathcal{S}_A$ is the smooth sextic curve
  in $\PP^2$ defined by
${\rm det}(A(y)) = y_1^6+y_2^6+y_3^6+y_1^2 y_2^2 y_3^2$.
The wave variety $\mathcal{W}_A$ is the smooth cubic curve in $\PP^2$ defined by
$z_1^3 - z_2^3 + z_3^3 - z_1z_2z_3$. These two plane curves are linked by 
 their incidence curve $\mathcal{I}_A \subset \PP^2 \times \PP^2$.
 If the entries of $A$ are replaced by random quadrics in $\p_1,\p_2,\p_3$,
 then $\mathcal{W}_A$ is a singular curve of degree $12$ in $\PP^2$.
\hfill $\triangle$ \end{example}

The article \cite{ADHR19} extended the results in
\cite{DPR}  by introducing two refined notions of wave~cones.
We now recast these as algebraic varieties.
 For $r \in \{0,\ldots,n-1\}$, the  \emph{level $r$ wave variety} is
\begin{equation}
\label{eq:unionofintersections}
  \mathcal{W}^r_A \,\,\, := \bigcup_{\pi \in \operatorname{Gr}(n-r, n)} \bigcap_{
    y \in \pi } \, \ker A(y).
\end{equation}
The union is over the Grassmannian
$\operatorname{Gr}(n-r, n)$ of linear subspaces $\pi$ of
codimension $r$ in $\PP^{n-1}$.
For basics on Grassmannians and their projective embeddings see
\cite[Chapter 5]{MS}.
 For $r=n-1$, the inner intersection in (\ref{eq:unionofintersections})
goes away, the outer union is over $y \in \PP^{n-1}$,
and we obtain the wave variety $ \mathcal{W}_A$.
At the other end of the spectrum, the level $0$ wave variety
$ \mathcal{W}^0_A \,\,\, = \bigcap_{y \in \PP^{n-1} }  \ker A(y)$ is often empty.
For the in-between levels $r$, we obtain a hierarchy 
\begin{equation}
\label{eq:wavehierarchy}
  \mathcal{W}^0_A \,\subseteq\, \mathcal{W}^1_A \,
  \subseteq \dotsb \,\subseteq \,\mathcal{W}^{n-1}_A
  \, =\, \mathcal{W}_A \,\,\subseteq\,  \, \PP^{k-1}.
\end{equation}

We now define a second hierarchy in $\PP^{k-1}$
by switching the intersections and the union.
Namely,  for any integer $r \in \{1,\ldots,n\}$, we define the
     \emph{level $r$ obstruction variety}  to be 
 \begin{equation}
  \label{eq:intersectionofunions}
    \mathcal{O}^r_A \,\,\, := \bigcap_{\sigma \in \operatorname{Gr}(r, n)} \bigcup_{
    y \in \sigma }\, \ker A(y).
 \end{equation}
  This intersection is over the Grassmannian 
   of $(r-1)$-dimensional subspaces in  $\PP^{n-1}$.
  The smallest and the largest obstruction variety
  coincides with the corresponding wave variety.
    
\begin{lemma}           \label{lem:inclusion}                      
We have  the inclusions $\mathcal{W}^r_A \subseteq \mathcal{O}^{r+1}_A$   for all $r$, with 
$\mathcal{W}^0_A = \mathcal{O}^1_A$ and                                                          
$\mathcal{W}^{n-1}_A = \mathcal{O}^n_A$.
\end{lemma}

\begin{proof}
Fix $z \in \mathcal{W}^r_A$
and a codimension $r$ subspace $\pi $ of $ \PP^{n-1}$
such that $A(y) z = 0$ for all $y \in \pi$.
Consider any $r$-dimensional subspace $\sigma $ of $\PP^{n-1}$.
Pick a point $w$ in the intersection $\pi \cap \sigma $.
Since $A(w) z = 0$, we~have $z \in \bigcup_{y \in \sigma} {\rm ker} A(y)$, 
and hence $z \in \mathcal{O}_A^{r+1}$.
Equality holds for $r = 0$ because $\mathcal{W}^0_A = \bigcap_{y \in \PP^{n-1}}
{\rm ker} A(y) = \mathcal{O}^1_A$, and for $r = n-1$ because
 $\mathcal{W}^{n-1}_A =
 \bigcup_{y \in \PP^{n-1}} {\rm ker} A(y) = \mathcal{O}^n_A$.
\end{proof}

In analogy to the wave varieties in (\ref{eq:wavehierarchy}),
there is also a   hierarchy of obstruction varieties:
\begin{equation}
\label{eq:obstructionhierarchy}
 \mathcal{W}^0_A \,= \,
  \mathcal{O}^1_A \,\subseteq\, \mathcal{O}^2_A \,\subseteq \dotsb \,\subseteq \,\mathcal{O}^n_A
  \, =\, \mathcal{W}_A \,\,\subseteq\,  \, \PP^{k-1}.
\end{equation}
    
\begin{example}[$n=3, k=4, r=2$] \label{ex:dadazwei}
Fix the matrix $A$ in Example \ref{ex:eins} and \ref{ex:zwei}.
For every $z \in \PP^3$, there exists $y \in \PP^2$  with $A(y) z = 0$, and hence
$ \mathcal{W}^2_A = \mathcal{O}^3_A = \PP^3$.
But, for every $z $, there also exists
$y \in \PP^2$ with $A(y) z \not=0 $, and hence
$ \mathcal{W}^0_A = \mathcal{O}^1_A = \emptyset$.
The variety in the middle of 
(\ref{eq:wavehierarchy}) and (\ref{eq:obstructionhierarchy}) satisfies
$\mathcal{W}^1_A = \mathcal{O}^2_A \subset \PP^3$. This
is the twisted cubic~curve 
$z = (s^3,s^2t,st^2,t^3)$.
Indeed, the  matrix
 $\binom{\,z_1 \,\, z_2  \,\, z_3 \,}{\, z_2  \,\,z_3 \,\, z_4 \,}$
 has rank $1$, with kernel
$\pi = \{ y \in \PP^2 \,: \, s^2 y_1 + st y_2 + t^2 y_3 = 0 \} \in {\rm Gr}(2,3)$.
 Every other line $\sigma \in {\rm Gr}(2,3)$ 
 in the projective plane $\PP^2$
 intersects the line $\pi$.
 \hfill $\triangle$
\end{example}

We next recall a basic construction from algebraic geometry; see
\cite[Example 6.19]{Harris}. Fix~a projective variety
  $X \subset \PP^{n-1}$.
The {\em Fano variety}
${\rm Fano}_r(X) $ is the subvariety of the Grassmannian
$ {\rm Gr}(n-r,n)$
whose points are the linear spaces $\pi$ of codimension $r$ in
$\PP^{n-1}$ that lie in $X$.
We use Fano varieties to argue that the inclusion 
in Lemma   \ref{lem:inclusion}  can be strict.

\begin{example}[$k=\ell=1, n \geq 3$]
A subvariety of $\PP^0$ is either empty or a point.
Let $A = [a]$ where $a$ is irreducible of degree $d$.
Then ${\rm Fano}_1(X) = \emptyset$. 
Our varieties in (\ref{eq:unionofintersections}) and
(\ref{eq:intersectionofunions})~are
$$ 
\mathcal{W}^r_A \, = \, \begin{cases}
\,\emptyset & {\rm if}\,\, {\rm Fano}_r(X) = \emptyset , \\
\, \PP^0 & {\rm if}\,\,   {\rm Fano}_r(X) \not= \emptyset,
\end{cases}
\qquad {\rm and} \qquad
\mathcal{O}^{r+1}_A \, = \, \begin{cases}
\,\emptyset & {\rm if}\,\, r=0, \\
\, \PP^0 & {\rm if}\,\, r \geq 1.
\end{cases}
 $$
 If $d \geq 2$ then ${\rm Fano}_1(X) = \emptyset$, so
     $\mathcal{W}^1_A$ is strictly contained in $\mathcal{O}^2_A$.
Equality holds for $d=1$.  \hfill $\triangle$
\end{example}

Returning to arbitrary $k$ and $\ell$, we now show
that equality always holds for first order PDE.
The main point for  $d=1$ is this: we can write
the product $\,A(y) z \,$ as $\, C(z) y \,$
where $C(z)$ is an $\ell \times n$-matrix whose
entries are linear forms in $z_1,\ldots,z_k$.
We did this in (\ref{eq:xiu}).

\begin{proposition} \label{prop:Cu} If $d=1$ then
$\mathcal{W}^r_A = \mathcal{O}^{r+1}_A = 
\bigl\{ z \in \PP^{k-1} \,: \, {\rm rank}(C(z)) \leq r \bigr\}\,$
  for all $r$.
\end{proposition}

\begin{proof}
Fix $z \in \PP^{k-1}$. The condition 
$z \in \mathcal{W}^r_A$ says that the
kernel of the matrix $C(z)$~contains a subspace $\pi$ of codimension $r$.
The condition 
$z \in \mathcal{O}^{r+1}_A$ says that the kernel of $C(z)$ 
meets every $r$-dimensional subspace $\sigma$ of $\PP^{n-1}$.
Both conditions are equivalent to
${\rm rank}(C(z)) \leq r$.
\end{proof}

Thus, the wave varieties of first order PDE
are easy to write down: they are 
the determinantal varieties of the auxiliary matrix $C(z)$.
For $d \geq 2$, elimination methods from nonlinear algebra (e.g.~Gr\"obner bases)
are needed to compute the defining equations of these varieties.

\begin{proposition} \label{prop:twocones}
The wave varieties $\mathcal{W}^r_A$ and the obstruction variety $\mathcal{O}^r_A$
are indeed varieties in
the projective space $\PP^{k-1}$, i.e.~they are zero
sets of homogeneous polynomials in $k$ variables.
\end{proposition}

\begin{proof}
The following incidence variety is closed in its ambient product space:
\begin{equation}
\label{eq:IrA} \mathcal{I}^r_A \,\, = \,\,
\bigl\{ \,(y,z,\pi) \in \PP^{n-1} \times \PP^{k-1} \times {\rm Gr}(n-r,n) \,:\,
A(y) z = 0 \,\,\, {\rm and} \, \,\,y \in \pi \,\bigr\}.
\end{equation}
The sets we defined in (\ref{eq:unionofintersections}) and
(\ref{eq:intersectionofunions})
are derived from this variety by quantifier elimination:
$$
\mathcal{W}^r_A  \, = \,
\bigl\{ \,z \,: \,\exists \pi \, \,\forall \,y \,\, (y,z,\pi) \in \mathcal{I}^r_A \,\bigr\} \quad {\rm and} \quad
\mathcal{O}^r_A  \, = \,
\bigl\{ \,z \,: \,\forall \pi \, \,\exists \,y \, \,(y,z,\pi) \in \mathcal{I}^r_A \,\bigr\} . 
$$
These two sets are closed in $\PP^{k-1}$ because all 
their defining equations are homogeneous in each group of variables.
For the existential quantifier this follows from the  Main Theorem of Elimination Theory
\cite[Theorem 4.22]{MS}.
For the universal quantifier once checks it directly.

We compute ideals for $\mathcal{W}^r_A$ and $\mathcal{O}^r_A$ as follows.
The equations $A(y) z= 0$ are bihomogeneous of degree $(d,1)$.
The condition $y \in \pi$ translates into 
 bilinear equations in $(y,p)$, where $p$ is the vector of Pl\"ucker coordinates of $\pi$.
We view these as equations in 
$y$ with coefficients in $(z,p)$, and we form the ideal of all coefficient 
polynomials. The zero set of this ideal is the subvariety
$\,\bigcap_{ y \in \pi } {\rm ker} A(y)$, which lies
 in ${\rm Gr}(r,n) \times \PP^{k-1}$.
We now project that variety onto the second factor to obtain $\mathcal{W}^r_A$.
This amounts to saturating and then eliminating the Pl\"ucker coordinates $p$.
What arises is an ideal in the unknowns $z$ whose zero set is $\mathcal{W}^r_A$.

To get the ideal of $\mathcal{O}^r_A$, we modify the argument as follows.
Again, we consider a fixed but unknown Pl\"ucker vector $p$ and we consider the
equations for $y \in \pi$ along with $A(y) z = 0$. From these equations
we eliminate $y$ to obtain polynomials in $(p,z)$ whose zero set is
$\bigcup_{y \in \pi } \ker A(y)$. We now vary $p$
and we view this as a subvariety of ${\rm Gr}(r,n) \times \PP^{k-1}$.
We consider the defining equations of this subvariety, 
and we write them as polynomials in $p$ whose
coefficients are polynomials in $z$. The collection of all such coefficient
polynomials defines a subvariety of $\PP^{k-1}$. By construction, that
subvariety equals the desired set $\mathcal{O}^r_A$.
\end{proof}

\section{Back to Analysis}
\label{sec5}

We now return to the setting of waves $\phi: \RR^n \rightarrow \CC^k$ 
that was introduced in Section \ref{sec2}.
The projective varieties $\mathcal{W}^r_A$ and $\mathcal{O}^r_A$
in $ \PP^{k-1}$ are to be viewed as affine cones in $\CC^{k}$.
We write
\begin{align*}
  \mathcal{W}_{A,\RR}^r \,\,\, &\coloneqq \bigcup_{\pi \in \Gr_\RR(n-r,n)} \bigcap_{y \in \pi \setminus \{0\}} \ker A(y), \\
  \mathcal{O}_{A,\RR}^r \,\,\, &\coloneqq \,\, \,\bigcap_{\sigma \in \Gr_\RR(r,n)}
  \, \,\bigcup_{y \in \sigma \setminus \{0\}} \ker A(y),
\end{align*}
where $\Gr_\RR(r,n)$ is the Grassmannian of $r$-dimensional subspaces in $\RR^n$.
In these definitions, 
the kernel of $A(y)$ is over the complex numbers, but
$\pi$ and $\sigma$ are required to be real. Hence 
$\mathcal{W}_{A,\RR}^r$ and $\mathcal{O}_{A,\RR}^r$ are subsets in $\CC^k$,
closely related to the projective varieties in
(\ref{eq:unionofintersections})~and~(\ref{eq:intersectionofunions}).

Readers of \cite{ADHR19} will note that we changed notation and 
nomenclature. The $\ell$-wave cone $\Lambda^\ell_\mathcal{A}$ from
\cite[Definition 1.2]{ADHR19} is the obstruction cone $\mathcal{O}^r_{A,\RR}$ here,
while the cone $\mathcal{N}^\ell_\mathcal{A}$ defined later in \cite[eqn (1.8)]{ADHR19}
is our wave cone $\mathcal{W}^r_{A,\RR}$.
The coming results  will motivate these choices.

Proposition~\ref{prop:wave_sols} shows why $\mathcal{W}^r_{A,\RR}$ 
serves as the $r$th wave cone.
The distribution in (\ref{eq:wavesolution}) has the form
$\, \RR^n \rightarrow \CC^k \,: x \,\mapsto \, \delta(L x) \cdot u $
where $L$ is the $(n-r) \times n$ matrix whose rows are 
the coefficients of $L_1,\ldots,L_{n-r}$.
Recall Remark~\ref{rmk:composition} for the definition of $\delta(Lx)$ as a distribution.

\begin{proposition}\label{prop:NN}
A vector $u\in \CC^k$ lies in the wave cone $\mathcal{W}^r_{A,\RR}$ 
if and only if there is a 
matrix $L \in \RR^{(n-r) \times n}$ such that $x \mapsto \delta(L x) \cdot u$ 
is a solution to $A$ for all
distributions  $\delta \in \mathcal{D}'( \RR^{n-r} , \CC)$.
\end{proposition}

\begin{proof}
By definition, 
a complex vector $u $ lies in the wave cone $ \mathcal{W}^r_{A,\RR}$ if and only if
there exists a real subspace $\pi \in {\rm Gr}_\RR(n-r,n)$ such that
$A(\xi) u = 0$ for all $\xi \in \pi \subseteq \RR^n$.
This is equivalent to saying that $(u, \pi)$ is a wave pair for $A$.
If we identify $\pi$ with the rowspace of $L$, then the result follows from Proposition~\ref{prop:wave_sols}.
\end{proof}

We next present an analogous statement for 
the obstruction cones $\mathcal{O}^r_{A,\RR}$.

\begin{proposition}\label{prop:obstruction}  A vector $u \in \CC^k$ lies in 
$\mathcal{O}^r_{A,\RR}$ if and only if, for all $S \in \RR^{r \times n}$ of rank $r$, the
PDE $A$ has a wave solution $\,x \mapsto \delta( Sx) \cdot u$
where $\delta$ is nonconstant and bounded.
\end{proposition}

\begin{proof}
Suppose $u \in \mathcal{O}^r_{A,\RR}$ and let $\sigma \in {\rm Gr}_\RR(r,n)$ 
be the real rowspan of the real matrix $S$. Fix a nonzero vector $\xi \in \sigma$ such that
$A(\xi)u = 0$, and let $\eta \in  \RR^r$ such that $\xi = \eta S$.
The exponential function $\delta_\eta (t) = {\rm exp}(i \eta \cdot t)$ is
nonconstant and bounded. Moreover, the function
$\delta_\eta(Sx) \cdot u$ is a wave solution to the PDE $A$, by the same calculation as in the proof of 
Proposition \ref{prop:NN}. This proves the only-if direction.

For the if-direction, let $u \notin \mathcal{O}^r_{A,\RR}$.
There exists $\sigma \in {\rm Gr}_{\RR}(r,n)$ such that $A(\xi)\cdot u\neq 0$ for all $\xi \in \sigma \backslash \{0\}$.
Let $S$ be as before the real matrix with rowspan $\sigma$.
Now suppose $\delta(Sx)\cdot u$ is a bounded solution of $A$.
By the proof of Proposition~\ref{prop:wave_sols}, this implies that $\delta(y)$ is a bounded solution of 
the operator $\alpha(\partial_y)=A(\partial_y S) \cdot u$.
This operator is elliptic by our assumption.
 By classical theory (cf.~\cite[Theorem 2-7]{schechter77}),
 every solution to $\alpha \bullet v=f$ with $f\in C^\infty$ is in $C^\infty$.
   Therefore a Liouville theorem holds: one can use the Closed Graph Theorem to deduce that there is a constant $C>0$ such that for any solution of $\alpha \bullet v =0$ in the unit ball $B_1$ one has 
\[ \|Dv\|_{L^\infty(B_{1/2})} \,\le\, C \|v\|_{L^\infty (B_1) }\,.\] 
Since the operator $\alpha$ is of homogenous degree $d$,  we can use scaling to obtain
\[ \|Dv\|_{L^\infty (B_R)}\, \,\le\, \frac{C}{R}\|v\|_{L^\infty(B_{2R})}. \]
Hence, every bounded solution on $\RR^{n-r}$ is constant
(cf.~\cite[Chapter 2]{schechter77}). So,
 $\delta$ is constant.
\end{proof}

We used the term ``obstruction'' for the variety $\mathcal{O}^r_A$
and the cone $\mathcal{O}^r_{A,\RR}$ not because their elements are
obstructions. Rather, our choice of name refers to role
played by the cone $\mathcal{O}^r_{A,\RR}$ in the paper \cite{ADHR19}
which motivated us.
Since $\mathcal{O}^{r}_{A,\RR}$
contains the wave cone $\mathcal{W}^{r-1}_{A,\RR}$, the
latter is empty if the former is empty. Thus, the cone
$\mathcal{O}^{r}_{A,\RR}$ being empty is an obstruction to~the
existence of wave solutions.
That obstruction is a key for the ``dimensional estimates'' in~\cite{ADHR19}.

In the present paper we often transition between real numbers and
 complex numbers. This occurs at multiple mathematical levels, including
trigonometry and projective geometry.
The complex numbers represent waves in Section \ref{sec2}
and they serve as an algebraically closed field in Section \ref{sec4}.
However, the argument $x$ of our solutions $\phi(x)$ are real vectors.
The spaces (\ref{eq:tempered}) belong to the field real analysis, as does
the study of $A$-free Radon measures in \cite{ADHR19, DPR, KR}.
Recall that a Radon measure is a distribution that admits an integral
representation, and one is interested in 
rectifiability of such measures that satisfy the PDE constraint given by~$A$.

This raises the question of how complex analysis fits in.
 From a purely algebraic point of view,
we can certainly consider solutions in the space of
holomorphic functions $\phi: \CC^n \rightarrow \CC^k$. 
All our formal results extend gracefully to that setting.
For instance, we can certainly take $\delta$ in 
(\ref{eq:wavesolution}) to be a holomorphic function on $\CC^{n-r}$
to get a holomorphic solution $\phi$ to our PDE.
However, from an analytic point of view, there are no
meaningful waves in complex analysis. The following example is
meant to illustrate the importance of reality for making waves.

\begin{example}[$n=2,k=\ell=1, \, d=1,2$]
We consider PDE for scalar-valued functions in two variables.
The {\em transport equation} $A = \p_1+\p_2$ has the solutions $\delta(x_1-x_2)$.
These are waves and $\delta$ can be any distribution. The
{\em Cauchy-Riemann equation} $A' = \p_1 + i \p_2$ looks very similar,
and we can write its solutions formally as $\delta(x_1 + ix_2)$. 
But, these solutions do not come from the wave cone $\mathcal{W}^r_{A,\RR}$,
since here $\pi$ is not real, and these do not give waves.
Passing to second order equations, one might compare
$\p_1^2- \p_2^2$ and $\p_1^2+\p_2^2$. These two PDE
look indistinguishable to the eyes of algebraist, while an analyst 
will see a {\em hyperbolic PDE} and an {\em elliptic PDE}. These
two classes have vastly different properties for their solutions.
In particular, the latter can only admit smooth solutions.
\hfill $\triangle$
\end{example}

The affine cones $\mathcal{W}_{A,\RR}^r$ and $\mathcal{O}_{A,\RR}^r$ can be quite different from the
complex varieties $\mathcal{W}_A^r$ and $\mathcal{O}_A^r$.
In general we have $\mathcal{W}_A^r \supseteq \mathcal{W}_{A,\RR}^r$.
Indeed, if $z \in \mathcal{W}_A^r$, there is a linear subspace $\pi \in \Gr(n-r,n)$ such that $A(y)z = 0$ for all $y \in \pi$.
For $z \in \CC^k$ to lie in $\mathcal{W}_{A,\RR}^r$, we must impose the additional condition that the dimension of $\pi \cap \RR^n$ is also $n-r$.
The inclusions for the obstruction cones are reversed: $\mathcal{O}_A^r \subseteq \mathcal{O}_{A,\RR}^r$.
The point $z \in \CC^k$ lies in $\mathcal{O}_A^r$ if and only if for all $\sigma \in \Gr(r,n)$ there exists $y\in \sigma \setminus \{0\}$ such that $A(y)z = 0$.
This condition is relaxed in $\mathcal{O}_{A,\RR}^r$, where it suffices to consider those $\sigma$ whose real part $\sigma \cap \RR^n$ also has dimension $r$.

We close with an example that highlights the connection
to the theory of rank-one convexity in the study of nonlinear PDE 
\cite{KMS}. Here one is interested in solutions to $A$
that additionally satisfy differential inclusions $\phi(x) \in \mathcal{K}$,
where $\mathcal{K}$ is a specified subset of $\CC^k$. Of
special interest in the case when the target is a matrix space and
$\mathcal{K}$ is a finite set of matrices.

\begin{example}[$n=2, d=1, k=3,\ell = 2$] \label{ex:curl1}
Consider the action of the curl operator on the $3$-dimensional space of
symmetric $2 \times 2$-matrices 
$\binom{\,\phi_1 \,\, \phi_2   \,}{\, \phi_2  \,\,\phi_3  \,}$. In our notation, this corresponds~to
$$ A \,=\, \begin{small}
\begin{bmatrix} -\p_2  & \! \phantom{-}\p_1 & \,0\,\, \\ \,\,0  & \! -\p_2 & \,\p_1 \,\,\end{bmatrix} .
\end{small} $$
This PDE is a simplified version of that in Examples \ref{ex:eins}, \ref{ex:zwei}
and \ref{ex:dadazwei}.
The wave cone consists of symmetric $2 \times 2$ matrices of rank $1$.
Here $\mathcal{K}$ is a finite set in $\RR^3$, such as the
five matrices in \cite{Pompe}, whose rank-one convex hull is of great interest.
Our varieties in Section \ref{sec4} offer an algebraic framework
for higher notions of convexity that might be of interest in analysis.
\hfill $\triangle$
\end{example}

\section{Computing Wave Pairs}
\label{sec6}

Our aim is to solve a PDE, given by 
an $\ell \times k$ matrix $A$ whose entries are homogeneous polynomials of degree $d$
in $R = \CC[\p_1,\ldots,\p_n]$.
Each wave (\ref{eq:wavesolution}) arises from a wave pair $(z,\pi)$, which serves as 
a blueprint for creating solutions to the PDE.
Our approach allows complete freedom in making waves with desirable analytic properties,
by choosing the distribution $\delta $ in Proposition \ref{prop:NN}.
Inspired by Proposition~\ref{prop:wave_sols}, we define the {\em wave pair variety} 
$$ \mathcal{P}^r_A \,\, = \,\,
\bigl\{ (z, \pi) \in \PP^{k-1} \times \Gr(n-r, n) \,\colon A(y)z = 0 \,\,\,
\hbox{for all}\, \,\,y \in \pi \bigr\}.
$$
This is a smaller version  of the incidence variety 
$\mathcal{I}^r_A$ we saw in (\ref{eq:IrA}).
The wave variety $\mathcal{W}^r_A$ introduced in
(\ref{eq:unionofintersections})
is the projection of the wave pair variety $\mathcal{P}^r_A$ onto the first factor $\PP^{k-1}$.
For $r=n-1$ the wave pair variety coindices with the incidence variety in (\ref{eq:IA}). In symbols,
\begin{equation}
\label{eq:IPn-1}
 \mathcal{P}^{n-1}_A \,\,=\,\, \mathcal{I}_A.
 \end{equation}

It is instructive to start with the
case $k=1$. Here $\mathcal{P}^r_A$ lives in
$\PP^0 \times \Gr(n-r,n)$, which we identify with
$\Gr(n-r,n)$. Consider the subvariety $\mathcal{S}_A$ 
of $\PP^{n-1}$ that is defined by the $\ell$ entries of the $\ell \times 1$ matrix $A$.
This is the support of our PDE, as seen in  (\ref{eq:SA}).
The condition 
$A(y) z = 0$ for $z \in \PP^0$
simply means that $ y \in \mathcal{S}_A$.
From this we conclude the following fact.

\begin{corollary} \label{cor:fano}
  If $k=1$ then $\mathcal{P}^r_A = {\rm Fano}_{r}(\mathcal{S}_A)$
is the Fano variety of the support $\mathcal{S}_A$. The points of
$\mathcal{P}^r_A$ are the linear spaces of
 codimension $r$ in $\PP^{n-1}$ that are contained in~$\mathcal{S}_A$.
 \end{corollary}

The software  {\tt Macaulay2}
has a built-in command {\tt Fano} 
for computing the ideal of the Fano variety
$ {\rm Fano}_{r}(\mathcal{S}_A)$ from the entries of $A$.
Our results in this section extend this method.
We shall describe an algorithm for computing $\mathcal{P}^r_A$ 
and all the varieties introduced in Section \ref{sec4}.

Each of our varieties lies in a projective space or product of projective spaces.
What we seek is its {\em saturated ideal}. To explain what this means,
consider the variety $\mathcal{I}_A $ in $ \PP^{n-1} \times \PP^{k-1}$.
Its description in (\ref{eq:IA}) is easy.
The $\ell$ coordinates of $A(y)z$ are
  polynomials of 
bidegree $(d,1)$~in 
$$ \CC[y,z] \,\, = \,\, \CC[y_1,\ldots,y_n,z_1,\ldots,z_k] . $$
However, these $\ell$ polynomials do not suffice. The saturated ideal
of the variety  $\mathcal{I}_A$ equals
\begin{equation}
\label{eq:satideal}
\bigl( \,\bigl( \,\langle \,A(y) z \,\rangle \,: \langle y_1,\ldots,y_n \rangle^\infty\, \bigr) :
\langle z_1,\ldots,z_k \rangle^\infty \,\bigr).
\end{equation}
Saturation is a built-in command in {\tt Macaulay2} \cite{M2},
but its execution often takes a long time.
This crucial step removes extraneous contributions by the
irrelevant ideals of $\PP^{n-1}$ and $\PP^{k-1}$.
 
\begin{example}[$k=\ell=n=d=2$]
The parameters are as in Example \ref{ex:222}, but now~the entries of
$A$ are general quadrics in $\CC[y_1,y_2]$.
The variety $\mathcal{I}_A$ consists four points in $\PP^1 \times \PP^1$.
Its ideal (\ref{eq:satideal})  is
generated by six polynomials of bidegrees $(0,4),(1,2),(1,2),(2,1),(2,1),(4,0)$.
The first and last equation are binary quartics that define the
 projections $\mathcal{S}_A $ and $\mathcal{W}_A$  into~$\PP^1$.
 These data encode the general solution to the PDE $A$.
For a concrete example, consider 
  \begin{align*}
    A \,=\, \begin{bmatrix}
      \partial_1^2 + 4\partial_2^2 & 17\partial_1\partial_2 \\
      2\partial_1\partial_2 & 4\partial_1^2 + \partial_2^2
    \end{bmatrix}.
  \end{align*}
  Here the general solution $\phi: \RR^2 \rightarrow \CC^2$
     is given by the following superposition of waves
  \begin{align*}
    \phi(x_1,x_2) = \begin{bmatrix}
      -17 \\ 4
    \end{bmatrix}\alpha(2x_1+x_2) + \begin{bmatrix}
      17 \\ 4
    \end{bmatrix}\beta(-2x_1+x_2) + \begin{bmatrix}
      -2 \\ 1
    \end{bmatrix}\gamma(x_1+2x_2) + \begin{bmatrix}
      2 \\ 1
    \end{bmatrix}\delta(x_1-2x_2),
  \end{align*}
  where $\alpha,\beta,\gamma,\delta \in \mathcal{D}'$.
  This can also be found using the methods described in \cite{AHS21}.
\hfill $\triangle$
\end{example}

The points $\pi$ in the Grassmannian $\Gr(n-r,n)$
will be represented as in \cite[Section 5.1]{MS}.
We write $\pi$ as the rowspace of
an $(n -r) \times n$ matrix $S = (s_{ij})$, that is,
$\pi = \{ wS \,: \,w \in \CC^{n-r}\}$. For a subset $I$
of cardinality $n-r$ in $\{1,\ldots,n\}$, 
the corresponding subdeterminant of $S$ is denoted $p_I$.
Then $p = (p_I) \in \CC^{\binom{n}{r}}$ is the
vector of {\em Pl\"ucker coordinates} of $\pi$.
The resulting embedding of $\Gr(n-r,n)$ into
$\PP^{\binom{n}{r}-1}$ is defined by the ideal $G$ of
{\em quadratic Pl\"ucker relations} \cite[Section 5.2]{MS}.
Subvarieties of $\Gr(n-r,n)$ are represented by saturated ideals in
$\CC[p]/G$.
In the special case $r=n-1$, we identify the Pl\"ucker coordinates
$p$ with $y=(y_1,\ldots,y_n)$.

The wave pair variety $\mathcal{P}_A^r$ lives in
$\PP^{k-1} \times \PP^{\binom{n}{r}-1}$. We shall compute
its saturated ideal in the polynomial ring $\CC[z,p]/G$.
A pair $(z,\pi)$ lies in $ \mathcal{P}_A^r$ if and only if
$A(wS)z = 0$ for all $w \in \CC^{n-r}$.
To express this in Pl\"ucker coordinates, we proceed as follows.
Write the $\ell$ entries of $A(wS)z$ as linear combinations of the
monomials $w^\alpha$, $\alpha \in \mathbb{N}^{n-r}$, with coefficients
 in $\CC[z,S]$. Let $\mathcal{J} $ be the ideal generated by these coefficients, and consider
  the ring map $\psi: \CC[z,p]/G \rightarrow \CC[z,S]/\mathcal{J}$
  which fixes each $z_i$ and  maps $p_I$ to the corresponding minor of~$S$.

\begin{algorithm}
  \caption{The ideal of the wave pair variety in Pl\"ucker coordinates.}
  \label{alg:pair_variety}
  \begin{algorithmic}
    \Require 
    A matrix $A \in \CC[y_1,\dotsc,y_n]_d^{\ell \times k}$ and an integer
    $r \in \{0,1,\ldots,n-1\}$
    \Ensure The saturated ideal in $\CC[z,p]/G$ that defines $\mathcal{P}_A^r$
    as a subvariety of $ \PP^{k-1} \times \PP^{\binom{n}{r}-1}$
    \State $S \gets (s_{ij})$, an $(n-r) \times n$ matrix whose entries are variables
    \State $\mathcal{J} \gets $ the ideal in $\CC[z,S]$ generated by the coefficients of the 
    monomials $w^\alpha$ in $A(wS)u$
    \State $G \gets $ the ideal of quadratic Pl\"ucker relations in $\CC[p]$,
    as described in    \cite[Section 5.1]{MS}
    \State $T \gets \CC[z,p]/ G$, the coordinate ring of the ambient space $\PP^{k-1} \times \Gr(n-r, n)$
    \State $\psi \gets $ the map from $T $ to $ \CC[z,S]/\mathcal{J}$ 
    that sends $p_I \mapsto {\rm det}(S_I)$ and $z_i \mapsto z_i$
    \State Compute $\mathcal{I} = \ker \psi$ and write its generators in the polynomial ring $\CC[z,p]$
    \State \Return the ideal saturation $((\mathcal{I} : \langle z \rangle^\infty) : \langle p \rangle^\infty)$,
    as in (\ref{eq:satideal}).
  \end{algorithmic}
\end{algorithm}

To compute the ideal of the wave variety $\mathcal{W}_A^r$, one can 
now eliminate the Pl\"ucker variables from the output of Algorithm~\ref{alg:pair_variety}.
This corresponds to projecting onto the first factor of $\mathcal{P}_A^r$.

We implemented Algorithm~\ref{alg:pair_variety} in {\tt Macaulay2}. For the code and its documentation see
$$ \hbox{
\url{https://mathrepo.mis.mpg.de/makingWaves}.}$$
Our command \texttt{wavePairs(A,r)} returns generators of the saturated ideal of $\mathcal{P}_A^r$ in $\QQ[z,p]/G$, where $G$ is the Pl\"ucker ideal, given by the built-in command \texttt{Grassmannian(n-r-1, n-1)}.
Since Algorithm \ref{alg:pair_variety}
generalizes the  computation of Fano varieties, running it can be slow. 
A common method for speeding this up is to restrict to an affine patch of the Grassmannian.
The optional argument \texttt{Patch => ...} implements this.
If $\texttt{Patch}$ is set to \texttt{true}, then the leftmost $(n-r) \times (n-r)$ submatrix of $S$ is the identity, as in
\cite[eqn (5.2)]{MS}. But the user can also select other charts by specifying a list of indices.

\smallskip

We now come to the special case of first-order PDE ($d=1$).
These are ubiquitous in applications, and computing the corresponding wave pair varieties is easier.
Here we use the $\ell \times n$-matrix $C(z)$
given by $\,A(y) z = C(z) y $.
The $\mathcal{W}^r_A$ are the determinantal varieties of~$C(z)$.

\begin{corollary} 
Let $d=1$, with notation as in Proposition \ref{prop:Cu}.
The wave pair variety equals
$$ \mathcal{P}^r_A \,\,  = \,\,
\bigl\{ \,(z, \pi) \in \PP^{k-1} \times \Gr(n-r, n) \,\,\colon \,
\pi \subseteq {\rm kernel}(C(z)) \,\bigr\}. $$
\end{corollary}

If $\pi$ is given as the row space of an
$(n-r) \times n $ matrix $S$ then 
$\pi \subseteq {\rm kernel}(C(z))$  means that
$C(z) \cdot S^T $ is the zero matrix of format $\ell \times (n-r)$.
Thus, $\mathcal{P}^r_A$ is a vector bundle over the wave
variety $\mathcal{W}^r_A$.
We shall explore these determinantal varieties for some scenarios
of geometric origin. These specify
PDE which admit interesting wave solutions
$x \mapsto \delta(Lx) \cdot u$. 

\begin{example}[Cubic Surfaces]
Every smooth cubic surface in $\PP^3$ is the determinant of a $3 \times 3$ matrix
of linear forms. The surface  contains $27$ lines, but that number can drop for
special cubics.
 We here present an example
with nine lines, namely Cayley's cubic surface:
$$ \qquad \qquad A \,\,= \,\, \begin{bmatrix}
\p_1 & \p_2 & \p_3 \\
\p_2 & \p_1 & \p_4 \\
\p_3 & \p_4 & \p_1 \end{bmatrix} \,\, ,\qquad 
 C \,\,= \,\, \begin{bmatrix} 
 z_1 & z_2 & z_3 & 0 \\
 z_2 & z_1 & 0 & z_3  \\
 z_3 &  0   & z_1 & z_2 \end{bmatrix} \qquad \qquad
(n=4, k=\ell=3).
$$
The only nontrivial wave variety consists of the six points in $\PP^2$
where $C(z)$ has rank $2$:
\begin{equation}
\label{eq:sixpoints}
\mathcal{W}^2_A \,= \,\mathcal{O}^3_A \,=\,
\bigl\{   \, (1:1:0), (1:-1:0), (1:0:1), (1: 0: -1), (0: 1: 1), (0: 1: -1) \,\bigr\}.
\end{equation}
The cubic surface $\mathcal{S}_A= \{\, y\in \PP^3 : {\rm det}(A(y)) = 0\, \}$
has four singular points. It is shown~in \cite[Figure 1.1]{MS}. Geometrically, $\mathcal{S}_A$
is the blow-up of $\PP^2$ at the six points  (\ref{eq:sixpoints}).
This is a general fact:
if $A$ is a $3 {\times} 3$ matrix representing a cubic surface then
its wave solutions are supported on the six lines, among   $27$, 
whose blow-down
maps the surface birationally onto~$\PP^2$.

The wave pair variety $\mathcal{P}^2_A $ lives in
$\PP^2 \times \PP^5$.  Its ideal is the output computed by
Algorithm~\ref{alg:pair_variety}:
$$ \begin{small} \begin{matrix}
\langle z_1,z_2{-}z_3,p_{14},p_{23},p_{24}{+}p_{34},p_{13}{-}p_{34},p_{12}{+}p_{34} \rangle \, \cap \,
\langle z_1,z_2{+}z_3,p_{14},p_{23},p_{24}{-}p_{34},p_{13}{+}p_{34},p_{12}{+}p_{34} \rangle \, \cap \,
\\
\langle z_2,z_1{-}z_3,p_{13},p_{24},p_{14}{+}p_{34},p_{23}{-}p_{34},p_{12}{-}p_{34} \rangle \, \cap \,
\langle z_2,z_1{+}z_3,p_{13},p_{24},p_{14}{-}p_{34},p_{23}{+}p_{34},p_{12}{-}p_{34} \rangle \, \cap \,
\\
\langle z_3,z_1{-}z_2,p_{12},p_{34},p_{14}{+}p_{24},p_{23}{+}p_{24},p_{13}{-}p_{24} \rangle \, \cap \,
\langle z_3,z_1{+}z_2,p_{12},p_{34},p_{14}{-}p_{24},p_{23}{-}p_{24},p_{13}{-}p_{24}\rangle.
\phantom{ \cap}
\end{matrix} \end{small}
$$
Its projection to $\PP^2$ is $\mathcal{W}^2_A$, while that to
$\PP^5$ yields six of the nine points in
${\rm Fano}_2(\mathcal{S}_A)$.
 \hfill $\triangle$
 \end{example}

We conclude by
  explicitly computing the wave pair varieties of certain operators that are prominent in the calculus of variations.
Such operators are built from ${\sf div}$, ${\sf curl}$, and ${\sf grad}$.
We refer to \cite[Example 2.1]{shankar_notes} for a warm-up from the control theory 
perspective of Section~\ref{sec3}.

Determinantal varieties are given by 
imposing rank constraints on matrices \cite[Lecture~9]{Harris}.
The following construction realizes
such varieties as wave cones of certain natural PDEs.

\begin{example}[Generic Determinantal Varieties]
Let ${\sf div} = (\p_1, \p_2, \ldots,\p_n)$, fix $p \geq 2$, 
and set $k= pn$, $\ell = p$. By taking the $p$-fold direct sum of ${\sf div}$,
we obtain the first order PDE
$$ A \quad = \quad  \begin{small}
\begin{bmatrix}
\,\,{\sf div} \,& 0 & \,\,\cdots \,\, & 0 \\
 0 &  \,{\sf div}\, & \cdots & 0 \\
 \vdots & & \ddots & \vdots \\
 0 & 0 & \cdots & \,{\sf div} \,\,\end{bmatrix} \end{small} $$
 for distributions $\phi : \RR^n \rightarrow \CC^{p \times n}$
 with coordinates $\phi_{ij}$, where $i=1,\ldots,p$ and $j=1,\ldots,n$.
 The matrix $C(z)$ defined by the bilinear equation $ A(y) z = C(z) y$ 
 has format $p \times n$. Its entries are distinct variables $z_{ij}$.
 The wave variety $\mathcal{W}_A^r \subset \PP^{pn-1}$ is the
 determinantal variety of all
 $p \times n$ matrices  $z$ of rank $\leq r$.
The wave pair variety $\mathcal{P}_A^r \subset \PP^{pn-1} \times \Gr(n-r,n)$
consists of pairs $(z,\pi)$ where $\pi$ 
 is in the kernel of $z$.
This is a {\em resolution of singularities} for
the determinantal variety $\mathcal{W}_A^r$. We refer to
Examples 12.1 and 16.18 in Harris' textbook \cite{Harris}.
  \hfill $\triangle$
 \end{example}

We next come to the curl operator, with its action on matrices as in \cite[Example 1.16~(c)]{KR}.
A first glimpse was seen in Example~\ref{ex:curl1}.
Fix any integer $n \geq 2$. We write {\sf curl} for the $ \binom{n}{2} \times n$ matrix 
whose rows are vectors $\partial_i e_j - \partial_j e_i$.
We take $A$ to be the $p$-fold direct sum of {\sf curl}.
This matrix has $\ell = p\binom{n}{2}$ rows and $k = pn$ columns. The following holds for this matrix $A$.

\begin{proposition} \label{prop:curl}
  Let $A$ be the curl operator for distributions $\phi:\RR^n \to \CC^{p\times n}$.
  The ideal of its wave pair variety $\mathcal{P}_A^{n-1} \subseteq \PP^{pn-1} \times \PP^{n-1}$ is 
  generated by the  $2 \times 2$ minors of the $(p{+}1) \times n$~matrix
\begin{equation}
\label{eq:augmentedmatrix}
    \begin{bmatrix}
        y_1 &  y_2 & \cdots & y_n \\
      z_{11} & z_{12} & \cdots & z_{1n} \\
      z_{21} & z_{22} & \cdots & z_{2n} \\
      \vdots & \vdots & \ddots & \vdots \\
      z_{p1} & z_{p2} & \cdots & z_{pn} 
      \end{bmatrix}.
\end{equation}
The wave variety $\mathcal{W}_A$ is similarly defined by the $2 \times 2$ minors of the $p \times n$ matrix $(z_{ij})$.
All other wave pair varieties $\mathcal{P}^r_A$ and wave varieties $\mathcal{W}^r_A$, indexed by $r \leq n-2$, are empty.
\end{proposition}

\begin{proof}
The ideal of the incidence variety $\mathcal{I}_A = \mathcal{P}_A^{n-1} $ is computed by the saturation
(\ref{eq:satideal}) from
$$ \langle \,A(y) z \,\rangle \,\,=\,\, \bigl\langle \,y_i z_{kj} - y_j z_{ki} \,: \,k = 1,\ldots,p \,\,\, {\rm and}\,\,\,
1 \leq i < j \leq n \,\bigr\rangle.
$$
This step removes contributions from the irrelevant maximal ideal.
Every $2 \times 2$ minor of (\ref{eq:augmentedmatrix}) lies in this saturated ideal.
Therefore, that ideal equals the prime ideal generated by all the $2 \times 2$ minors of (\ref{eq:augmentedmatrix}).
For $r \leq n-2$, we note that $A(y)z = C(z)y$ hands us the matrix
$$ C(z) \,\,\,= \,\,\, - \begin{bmatrix} {\sf curl}(z_{11},\ldots,z_{1n})  \\
 {\sf curl}(z_{21},\ldots,z_{2n}) \\ 
\cdots \qquad \cdots \\
{\sf curl}(z_{p1},\ldots,z_{pn}) \end{bmatrix}.
$$
One checks that this $p \binom{n}{2} \times n$  matrix cannot have rank $\leq n-2$ unless $z_{ij} = 0$ for all $i,j$.
\end{proof}

 \bigskip

\noindent
\footnotesize 
{\bf Authors' addresses:}

\smallskip

\noindent Marc H\"ark\"onen, Georgia Institute of Technology
\hfill {\tt harkonen@gatech.edu}

\noindent Jonas Hirsch, Universit\"at Leipzig
\hfill {\tt Jonas.Hirsch@math.uni-leipzig.de}

\noindent Bernd Sturmfels,
MPI-MiS Leipzig and UC Berkeley 
\hfill {\tt bernd@mis.mpg.de}
\end{document}